\documentclass[a4paper, 12pt, reqno]{amsart}

\usepackage[english]{babel}

\usepackage{amsfonts, amsmath, amssymb, amsthm}
\usepackage[shortlabels]{enumitem}
\setlist[enumerate]{label=\normalfont{(\arabic*)}}
\usepackage[colorlinks=true, linkcolor=blue, citecolor=blue]{hyperref}
\usepackage{stmaryrd}
\usepackage{tikz-cd}

\usepackage{xifthen}

\advance\hoffset-20mm
\advance\textwidth40mm

\newcommand{\bb}{\mathbb}

\newcommand{\der}{\partial}
\newcommand{\Ga}{\bb{G}_\mathrm{a}}

\DeclareMathOperator{\A}{A}
\DeclareMathOperator{\Cl}{Cl}

\DeclareMathOperator{\Fix}{Fix}
\DeclareMathOperator{\mIm}{Im}
\renewcommand{\Im}{\mIm}
\DeclareMathOperator{\Ker}{Ker}
\DeclareMathOperator{\LND}{LND}
\DeclareMathOperator{\pl}{pl}
\DeclareMathOperator{\Spec}{Spec}

\newcommand{\insertText}[2][]{%
    \ifthenelse{\isempty{#1}}{%
        \quad \text{#2} \quad%
    }{%
        #1 \text{#2} #1%
    }%
}
\newcommand{\restrict}[2]{
    \left.
    \kern-\nulldelimiterspace{#1}
    \right|_{#2}
}

\newtheorem{counter}{}[section]

\theoremstyle{definition}
\newtheorem{defi}[counter]{Definition}
\newtheorem{exam}[counter]{Example}

\theoremstyle{plain}
\newtheorem{coro}[counter]{Corollary}
\newtheorem{lemm}[counter]{Lemma}
\newtheorem{prop}[counter]{Proposition}
\newtheorem{theo}[counter]{Theorem}

\begin{document}

%%%%%%%%%%%%%%%%%%%%%%%%%%%%%%

\date{}
\title[Cylinders and the zero locus of the plinth ideal]{Cylinders and the zero locus of the plinth ideal}

\author[Kirill Shakhmatov]{Kirill Shakhmatov}
\address{HSE University, Faculty of Computer Science, Pokrovsky Boulevard 11, Moscow, 109028 Russia}
\email{kshahmatov@hse.ru}

\thanks{The study was implemented in the framework of the Basic Research Program at HSE University (HSE-BR-2025-84).}


%%%%%%%%%%%%%%%%%%%%%%%%%%%%%%

\begin{abstract}
Given a $\mathbb{G}_\mathrm{a}$-action on an affine variety $X$, we show that the complement of the union of all principal invariant cylinders in $X$ is equal to the zero locus of the plinth ideal of the corresponding locally nilpotent derivation.
\end{abstract}

\maketitle

%%%%%%%%%%%%%%%%%%%%%%%%%%%%%%

\section{Introduction}

%%%%%%%%%%%%%%%%%%%%%%%%%%%%%%

We fix an algebraically closed field $\bb{K}$ of characteristic zero. By a variety we mean an integral separated scheme of finite type over $\bb{K}$. Denote by $\Ga = (\bb{K}, +)$ the one-dimensional unipotent algebraic group.

The study of $\Ga$-actions on algebraic varieties is a long-standing subject in algebraic geometry and commutative algebra. In addition to the fact that a description of $\Ga$-actions is an object of study in itself, the algebraic nature of $\Ga$-actions on affine varieties provides a technique that allows to tackle various problems. One example is the study of automorphism groups \cite{M-L01}, \cite{Ga21}. Another is the Zariski Cancellation Problem. Exponential maps -- a description of $\Ga$-actions in the case of positive characteristic -- were used by Gupta in \cite{Gu14} to provide counterexamples to the Zariski Cancellation problem in positive characteristic.

The results of the papers \cite{KPZ11} and \cite{KPZ13} provide a way to describe $\Ga$-actions using cylinders, that is, open subsets of the form $\bb{A}^1 \times Z$. In particular, this technique can be used to study
\linebreak
$\Ga$-actions on affine cones over a projective variety $X$ in terms of internal geometric properties of $X$. This method was used in \cite{CPW16} to solve a long-standing problem about absence of non-trivial $\Ga$-actions on affine Fermat cubic threefold.

Let $X$ be an affine variety. Then $\Ga$-actions on $X$ are in one-to-one correspondence with locally nilpotent derivations on the algebra $\bb{K}[X]$ of regular functions on $X$. Given a locally nilpotent derivation $\der$ on a $\bb{K}$-algebra $B$, the ideal $\pl(\der) = \Ker \der \cap \Im \der$ in $\Ker \der$ is called the plinth ideal of $\der$. Factorial affine varieties with $\Ga$-actions having principal plinth ideal were studied in~\cite{Ma23}. In this note we obtain a geometric description of the zero locus $\bb{V}_X(\pl(\der))$ of the plinth ideal.

Let $\alpha \colon \Ga \times X \to X$ be a $\Ga$-action on an affine variety $X$. By \cite[Proposition~3.1.5]{KPZ11} there exists a principal cylinder $U \cong \bb{A}^1 \times Z$ in $X$ such that every $\bb{A}^1$-fiber of $U$ is an orbit of $\alpha$. We are interested in the geometric locus of the union of such cylinders. One can view this union as the locus of points in $X$ where the quotient $X \to Y$, $Y = \Spec(\Ker \der)$ admits a local principal trivialization of the action $\alpha$.

This note consists of two sections. The main result of the first section is Theorem~\ref{CPI.th}, in which it is proved that the union of principal cylinders defined by an action $\alpha$ coincides with the complement of the zero locus $\bb{V}_X(\pl(\der))$. In the second section we provide counterexamples to several natural assumptions, showing that this description does not admit a direct simplification.

The author is grateful to Ivan Arzhantsev and Sergey Gaifullin for useful discussions and suggestions.

%%%%%%%%%%%%%%%%%%%%%%%%%%%%%%

\section{Cylinders and the plinth ideal}

%%%%%%%%%%%%%%%%%%%%%%%%%%%%%%

Let us introduce basic definitions and conventions.

\smallskip

Let $X$ be a variety and $S \subseteq \bb{K}[X]$ be a set of regular functions on $X$. Denote by $\bb{V}_X(S)$ the closed subvariety in $X$ given by equations $h = 0$ for $h \in S$. Denote $\bb{D}_X(S) = X \setminus \bb{V}_X(S)$.

Let $G$ be an algebraic group and $X$ be a variety. By an action of $G$ on $X$ we mean a regular action $G \times X \to X$. Consider an action $\alpha \colon G \times X \to X$. We denote by $\Fix(\alpha)$ the set of fixed points of the action $\alpha$. Given an invariant subvariety $Y \subseteq X$, we denote by $\restrict{\alpha}{Y}$ the induced action $G \times Y \to Y$.

Let $B$ be a $\bb{K}$-algebra. A \emph{locally nilpotent derivation} (abbreviated \emph{LND}) on $B$ is a
\linebreak
$\bb{K}$-derivation $\der \colon B \to B$ satisfying the following property: for each $b \in B$ there exists $n \in \bb{N}$ such that $\der^n(b) = 0$. The set of all locally nilpotent derivations on $B$ is denoted by $\LND(B)$.

\smallskip

Let us recall an algebraic description of $\Ga$-actions on affine varieties. Consider an affine variety $X$. If $\alpha \colon \Ga \times X \to X$ is a $\Ga$-action on $X$, then the composition
$$
\begin{tikzcd}
\bb{K}[X] \ar{r}{\alpha^*} & \bb{K}[X][s] \ar{r}{d / d s} &
\bb{K}[X][s] \ar{r}{s = 0} & \bb{K}[X]
\end{tikzcd}
$$
is a locally nilpotent derivation $\LND(\alpha) \in \LND(\bb{K}[X])$. Conversely, if $\der \colon \bb{K}[X] \to \bb{K}[X]$ is a locally nilpotent derivation, then the map
$$
\exp(s \der) =
\sum_{n = 0}^\infty \frac{s^n \der^n}{n!} \colon
\bb{K}[X] \to \bb{K}[X][s]
$$
defines an action $\A(\der) \colon \Ga \times X \to X$. This correspondence between $\Ga$-actions on $X$ and locally nilpotent derivations on $\bb{K}[X]$ is bijective.

The following lemma is a simple technical observation that we will need later.

% Gluing Derivations
\begin{lemm} \label{GD.le}
Consider a variety $X$ and a $\Ga$-action $\alpha$ on $X$. Let $U_1, U_2 \subseteq X$ be two
\linebreak
$\alpha$-invariant affine open subsets and let $f \in \bb{K}(X)$ be a rational function on $X$, regular on $U_1 \cup U_2$. For each $i = 1, 2$ denote $\der_i = \LND \left( \restrict{\alpha}{U_i} \right) \in \LND(\bb{K}[U_i])$. Then $\der_1(f) = \der_2(f)$ as rational functions on $X$.
\end{lemm}

\begin{proof}
First we deal with the case $U_2 \subseteq U_1$. Consider a diagram
$$
\begin{tikzcd}
\Ga \times U_2 \ar{r} \ar{d} & U_2 \ar{d}
\\
\Ga \times U_1 \ar{r} & U_1
\end{tikzcd}
,
$$
where the horizontal arrows are restrictions of the action $\alpha$ and the vertical arrows are natural embeddings. Clearly, this diagram is commutative, hence the dual diagram
$$
\begin{tikzcd}
\bb{K}[U_1] \ar{r} \ar{d} & \bb{K}[U_1][s] \ar{d}
\\
\bb{K}[U_2] \ar{r} & \bb{K}[U_2][s]
\end{tikzcd}
$$
is also commutative. Therefore, we have $\der_2(f) = \restrict{\der_1(f)}{U_2}$ by construction.

Now consider the general case $U_1, U_2 \subseteq X$ and denote $U_3 = U_1 \cap U_2$. Then $U_3 \subseteq X$ is an
\linebreak
$\alpha$-invariant affine open subset, so there exists a locally nilpotent derivation $\der_3 = \LND \left( \restrict{\alpha}{U_3} \right) \in$
\linebreak
$\in \LND(\bb{K}[U_3])$. By the first case we have
$$
\der_3(f) = \restrict{\der_1(f)}{U_3}
\insertText{and}
\der_3(f) = \restrict{\der_2(f)}{U_3},
$$
hence $\der_1(f) = \der_2(f)$ in $\bb{K}(X)$.
\end{proof}

\begin{defi}
Consider a variety $X$. A \emph{cylinder in $X$} is an open subset $U$ together with an isomorphism $U \cong \bb{A}^1 \times Z$ for some variety $Z$. By a \emph{fiber} of a cylinder $U$ we mean a fiber of the projection $U \to Z$ to the second factor. We call the projection $U \to \bb{A}^1$ to the first factor the \emph{level function} of the cylinder $U$.
\end{defi}

\begin{defi}
Let $\alpha \colon \Ga \times X \to X$ be a $\Ga$-action on a variety $X$. We say that $\alpha$ is \emph{cylindrical} if there is an isomorphism $X \cong \bb{A}^1 \times Z$ for some variety $Z$ such that $\alpha$ acts by translations along the first factor.
\end{defi}

Let $B$ be a $\bb{K}$-algebra and $\der \in \LND(B)$. An element $b \in B$ is called a \emph{slice} of $\der$ if $\der(b) = 1$. Assume that $X$ is an affine variety, $B = \bb{K}[X]$ and $\alpha = A(\der)$ is the corresponding $\Ga$-action on $X$. Then by Slice Theorem \cite[Corollary~1.26]{Fr17} the action $\alpha$ is cylindrical if and only if the derivation $\der$ admits a slice. Moreover, a function $f \in \bb{K}[X]$ is a slice of $\der$ if and only if $f$ is a level function of a cylinder $X \cong \bb{A}^1 \times Z$ and $\alpha$ is cylindrical with respect to this isomorphism.

The following result is a generalization of \cite[Lemma~2.3]{ST25} for arbitrary varieties.

% Cylinders of \Ga-Actions
\begin{prop} \label{CGA.pr}
Let $X$ be a variety and $\alpha$ be a $\Ga$-action on $X$. Assume that there is an isomorphism $X \cong \bb{A}^1 \times Z$ such that every $\bb{A}^1$-fiber is an orbit of $\alpha$. Then $\alpha$ is cylindrical.
\end{prop}

\begin{proof}
Consider a covering $Z = \bigcup_{i \in I} U_i$ of $Z$ by open affine charts. For each $i \in I$ denote $V_i = \bb{A}^1 \times U_i$. Then $X = \bigcup_{i \in I} V_i$ is a covering of $X$ by affine $\alpha$-invariant cylinders. For each $i \in I$ denote
$$
\der_i =
\LND \left( \restrict{\alpha}{V_i} \right) \in
\LND(\bb{K}[V_i]).
$$
Denote by $s$ the level function of the cylinder $X \cong \bb{A}^1 \times Z$. By Lemma~\ref{GD.le} the functions $h_i = \der_i(s) \in \bb{K}[V_i] \subseteq \bb{K}(X)$ can be glued together, defining a function $h \in \bb{K}[X]$. By the argument from \cite[Lemma~2.3]{ST25}, applied to each chart $V_i$, we have $h \in \bb{K}[Z] \subseteq \bb{K}[X]$, the function $h$ is invertible and
$$
\alpha \big( s, (u, z) \big) = (u + s h, z)
$$
for all $s \in \Ga$ and $(u, z) \in \bb{A}^1 \times Z$. Conjugating the action $\alpha$ by the isomorphism
$$
\phi \colon \bb{A}^1 \times Z \to \bb{A}^1 \times Z, \quad
\phi(u, z) = \left( \frac{u}{h}, z \right),
$$
we obtain an action
$$
s \cdot (u, z) =
\phi \Big( \alpha \big( s, \ \phi^{-1}(u, z) \big) \Big) =
\phi \Big( \alpha \big( s, \ (u h, z) \big) \Big) =
\phi(u h + s h, z) = (u + s, z).
$$
Therefore, $\alpha$ is cylindrical.
\end{proof}

Let $X$ be an affine variety. By \cite[Proposition~3.1.5]{KPZ11} every $\Ga$-action on $X$ gives rise to a principal cylinder in $X$. Namely, if $\alpha \colon \Ga \times X \to X$ is a $\Ga$-action on $X$, then there exists an $\alpha$-invariant principal cylinder $U \subseteq X$ such that $\restrict{\alpha}{U}$ is cylindrical.

\begin{defi}
Let $\alpha \colon \Ga \times X \to X$ be a $\Ga$-action on a variety $X$. By an \emph{$\alpha$-cylinder} we mean an $\alpha$-invariant cylinder $U \subseteq X$ such that $\restrict{\alpha}{U}$ is cylindrical.
\end{defi}

Let $B$ be a $\bb{K}$-algebra and $\der \in \LND(B)$. The \emph{plinth ideal} of the derivation $\der$ is the set
$$
\pl(\der) = \Ker(\der) \cap \Im(\der).
$$
This is an ideal in the algebra $\Ker(\der)$.

The following result is a description of the set of functions $h \in \bb{K}[X]$ that define a principal $\alpha$-cylinder $\bb{D}_X(h)$.

% Cylinders and Plinth Ideal
\begin{theo} \label{CPI.th}
Let $X$ be an affine variety, $\alpha$ be a $\Ga$-action on $X$, and $h \in \bb{K}[X]$. Denote $\der = \LND(\alpha) \in \LND(\bb{K}[X])$. Then $\bb{D}_X(h)$ is an $\alpha$-cylinder if and only if $h^n \in \pl(\der)$ for some $n \in \bb{N}$.
\end{theo}

\begin{proof}
Assume that there exists $n \in \bb{N}$ such that $h^n \in \pl(\der)$. Since $\bb{D}_X(h) = \bb{D}_X(h^n)$, we may assume that $h \in \pl(\der)$. The assertion now follows from the construction in \cite[Proposition~3.1.5]{KPZ11}.

Conversely, denote $U = \bb{D}_X(h)$ and assume that $U \cong \bb{A}^1 \times Z$ is an $\alpha$-cylinder. The LND $\der$ extends to an LND on $\bb{K}[U] \supseteq \bb{K}[X]$. Then, since $h$ is an invertible element of the algebra $\bb{K}[U]$, we have $\der(h) = 0$. Moreover,
$$
\bb{K}[U] = \bb{K}[X]_h \cong \bb{K}[Z][s],
$$
where $\bb{K}[Z] = \Ker \der$ and $\der(s) = 1$. There exist $f \in \bb{K}[X]$ and $n \in \bb{N}$ such that
$
s = \frac{f}{h^n}.
$
Therefore,
$$
h^n s \in \bb{K}[X]
\insertText{and}
\der(h^n s) = h^n \in \pl(\der)
$$
as required.
\end{proof}

% Union of Cylinders
\begin{coro} \label{UC.co}
Let $X$, $\alpha$ and $\der$ be as in Theorem~\ref{CPI.th}. Let $\{U_i\}_{i \in I}$ be the set of all principal $\alpha$-cylinders in $X$. Then
$$
X \setminus \bigcup_{i \in I} U_i =
\bb{V}_X \big( \pl(\der) \big).
$$
\end{coro}

It is worth to mention here the equality $\bb{V}_X(\Im \der) = \Fix(\alpha)$.

% Principal Plinth Ideal
\begin{coro} \label{PPI.co}
Let $X$, $\alpha$ and $\der$ be as in Theorem~\ref{CPI.th}. Assume that $\pl(\der) \subseteq \Ker(\der)$ is a principal ideal. Then there exists a maximal principal $\alpha$-cylinder in $X$.
\end{coro}

Note that by a result of Bonnet, Daigle, and Kaliman \cite[Theorem~1.5]{DK09} if $X = \bb{A}^3$, then $\pl(\der)$ is a principal ideal in $\Ker(\der)$.

%%%%%%%%%%%%%%%%%%%%%%%%%%%%%%

\section{Examples}

%%%%%%%%%%%%%%%%%%%%%%%%%%%%%%

The first example is a $\Ga$-action $\alpha$ on $\bb{A}^3$ such that $\bb{A}^3 \setminus \Fix(\alpha)$ is not an $\alpha$-cylinder. Note that there is no such action on $\bb{A}^2$. Indeed, by a result of Rentschler \cite{Re68} every $\Ga$-action on $\bb{A}^2$ is conjugate to an action of the form $s \cdot (x, y) = (x + s p, y)$ for some $p \in \bb{K}[y]$. Then $\Fix(\alpha) = \bb{V}_{\bb{A}^2}(p)$ and $\bb{A}^2 \setminus \Fix(\alpha)$ is a cylinder $\bb{A}^1 \times \bb{D}_{\bb{A}^1}(f)$.

% Fixed Points
\begin{exam} \label{FP.ex}
Consider an LND
$$
\der = y \frac{\der}{\der x} + z \frac{\der}{\der y} \in
\LND(\bb{K}[x, y, z]).
$$
The corresponding $\Ga$-action $\alpha$ on $\bb{A}^3$ is
$$
s \cdot (x, y, z) =
\left( x + s y + \frac{1}{2} s^2 z, \ y + s z, \ z \right).
$$
Denote $\bb{K}^\times = \bb{K} \setminus \{0\}$. The orbits of $\alpha$ are parabolas
$$
z = \lambda, \quad
x = \frac{1}{2 \lambda} y^2 + a
$$
for some $\lambda \in \bb{K}^\times$ and $a \in \bb{K}$; lines
$$
z = 0, \quad y = \lambda
$$
for some $\lambda \in \bb{K}^\times$; and fixed points
$$
z = 0, \quad y = 0, \quad x = a
$$
for some $a \in \bb{K}$. One can check that
$$
\Ker(\der) = \bb{K}[z, \ y^2 - 2 x z]
\insertText{and}
\pl(\der) = \Ker(\der) \cdot z.
$$
Although $U = \bb{A}^3 \setminus \Fix(\alpha)$ is a cylinder $U \cong \bb{A}^1 \times \big( \bb{A}^2 \setminus \{(0, 0)\} \big)$, we claim that $U$ is not an $\alpha$-cylinder. Assume the contrary that there is an isomorphism $\phi \colon U \to \bb{A}^1 \times Z$ such that $\restrict{\alpha}{U}$ is cylindrical with respect to $\phi$. Denote by $f$ the level function of $\phi$. Since $\Fix(\alpha)$ has codimension greater than one in $\bb{A}^3$, we have $f \in \bb{K}[\bb{A}^3]$. Consider an $\alpha$-invariant subset $V = \bb{D}_{\bb{A}^3}(z)$. The isomorphism $\phi$ restricts to an isomorphism $\restrict{\phi}{V} \colon V \to \bb{A}^1 \times Z'$ with a level function $\restrict{f}{V}$. Therefore, $\restrict{f}{V}$ is a slice of $\der$ on $\bb{K}[V] = \bb{K}[x, y, z, z^{-1}]$. It is easy to see that every such slice has the form
$$
\frac{y}{z} + h(z, \ z^{-1}, \ y^2 - 2 x z)
$$
for some polynomial $h$, but these functions are not regular on $\bb{A}^3$, which is a contradiction.
\end{exam}

The second example is a free $\Ga$-action $\alpha$ on an affine surface $X$ such that $X$ cannot be covered by $\alpha$-cylinders. Again, there is no such action on $\bb{A}^2$. We will need the following description of cylinders in $\bb{P}^2$.

% Cylinders in P^2
\begin{lemm} \label{CP2.le}
Let $U \subseteq \bb{P}^2$ be a cylinder $U \cong \bb{A}^1 \times C$. Then $U$ is affine and $C$ is isomorphic to an open subset of $\bb{A}^1$.
\end{lemm}

\begin{proof}
Clearly, $C$ is a smooth curve. We claim that $C$ is rational. Indeed, the open embedding $U \subseteq \bb{P}^2$ induces a surjection of divisor class groups $\Cl(\bb{P}^2) \to \Cl(U)$. Since $\Cl(\bb{P}^2) \cong \bb{Z}$ and $\Cl(U) \cong \Cl(C)$, we see that $\Cl(C)$ is a quotient group of $\bb{Z}$. In particular, $\Cl(\hat{C})$ is a finitely generated group, where $\hat{C}$ is a smooth projective model of $C$. It follows that the genus of $\hat{C}$ is zero, hence $C$ is rational.

A smooth rational curve is isomorphic to an open subset of $\bb{P}^1$, so it suffices to show that $C \ne \bb{P}^1$. But if $C = \bb{P}^1$, then the curves $\{0\} \times \bb{P}^1$ and $\{1\} \times \bb{P}^1$ in $U$ are disjoint and closed in $\bb{P}^2$, which is a contradiction.
\end{proof}

\begin{exam}
Consider a $\Ga$-action $\alpha$ on $\bb{P}^2$ defined by the formula
$$
s \cdot (x : y : z) = \left(
    x + s y + \frac{1}{2} s^2 z \ : \ y + s z \ : \ z
\right).
$$
One can obtain this action from the action in Example~\ref{FP.ex} by passing to the quotient of $\bb{A}^3$ by the standard action of scalar matrices. The only fixed point of $\alpha$ is $(1 : 0 : 0)$ and the closures of other orbits of $\alpha$ are quadrics
$$
x z = \frac{y^2}{2} + a z^2
$$
for $a \in \bb{K}$ and the line $z = 0$. The affine variety
$$
X = \bb{P}^2 \setminus \{y^2 = 2 x z\}
$$
is $\alpha$-invariant and the induced action $\restrict{\alpha}{X}$ has no fixed points.

Denote $h = y^2 - 2 x z$. The algebra of regular functions on $X$ is the subalgebra in $\bb{K}(x, y, z)$ of functions of the form $f / h^n$, where $n$ is a non-negative integer and $f \in \bb{K}[x, y, z]$ is a homogeneous polynomial of degree $2 n$. One check that
$$
\Ker(\der) = \bb{K}[z^2 / h]
\insertText{and}
\pl(\der) = \Ker(\der) \cdot \frac{z^2}{h}.
$$
Let us show that every cylinder $U \subseteq X$ is a principal proper subset. First, $X$ is not cylinder, since by Lemma~\ref{CP2.le} we have $\Cl(U) = 0$ but $\Cl(X) \cong \bb{Z} / 2 \bb{Z}$. Therefore, $U \subsetneq X$ is a proper subset. The variety $X$ is smooth and the cylinder $U$ is affine by Lemma~\ref{CP2.le}, so the complement $X \setminus U$ is of pure codimension one. Since $\Cl(X) \cong \bb{Z} / 2 \bb{Z}$ we conclude that $U$ is principal. Thus, from Corollary~\ref{UC.co} we deduce that the union of all $\alpha$-cylinders in $X$ is the proper subset $z \ne 0$.
\end{exam}

In the third example we show that for affine $X$ there may be an affine $\alpha$-cylinder $U \subseteq X$ which cannot be covered by principal $\alpha$-cylinders.

\begin{exam}
Consider a $\Ga$-action on $\bb{A}^3$ defined by an LND $y \frac{\der}{\der x} + z \frac{\der}{\der y}$ as in Example~\ref{FP.ex}. The Danielewski surface
$$
X = \{ y^2 - 2 x z - 1 = 0 \}
$$
is invariant with respect to this action. Denote by $\alpha$ the restriction of this action to $X$ and let $\der = \LND(\alpha) \in \LND(\bb{K}[X])$. Let $U \subseteq X$ be the complement to the orbit $z = y - 1 = 0$. The extension of $\der$ to $\bb{K}[U]$ admits a slice $\frac{y + 1}{z} = \frac{2 x}{y - 1}$, hence $U$ is an $\alpha$-cylinder
$$
\begin{array}{rcl}
U & \cong & \bb{A}^1 \times \bb{A}^1
\\
(x, y, z) & \mapsto & \left( z, \frac{y + 1}{z} \right)
\\
\left( \frac{ (u v - 2) v }{2}, \ u v - 1 , \ u \right) &
\mapsfrom & (u, v)
\end{array}
.
$$
On the other hand, we see that
$$
\Ker(\der) = \bb{K}[z]
\insertText{and}
\pl(\der) = \Ker(\der) \cdot z.
$$
By Corollary~\ref{UC.co} the union of all principal $\alpha$-cylinders in $X$ is equal to
$$
\bb{D}_X(z) = U \setminus \{ z = y + 1 = 0 \}.
$$
Therefore, $U$ cannot be covered by principal $\alpha$-cylinders.
\end{exam}

The final example is a $\Ga$-action $\alpha$ on $\bb{A}^4$ such that $\bb{D}_{\bb{A}^4} \big( \pl(\der) \big)$ is not an $\alpha$-cylinder. By Corollaries~\ref{UC.co} and~\ref{PPI.co} we would need the plinth ideal $\pl(\der)$ to be non-principal, hence no such actions occur for $\bb{A}^3$.

\begin{exam}
Consider an LND
$$
\der = u \frac{\der}{\der x} + v \frac{\der}{\der y} \in
\LND(\bb{K}[x, y, u, v]).
$$
The corresponding $\Ga$-action $\alpha$ on $\bb{A}^4$ is
$$
s \cdot (x, y, u, v) = (x + s u, \ y + s v, \ u, \ v).
$$
One can check that
$$
\Ker(\der) = \bb{K}[u, \ v, \ x v - y u]
\insertText{and}
\pl(\der) = \Ker(\der) \cdot u + \Ker(\der) \cdot v.
$$
Using an argument similar to the one from Example~\ref{FP.ex}, we can show that the cylinder $U =$
\linebreak
$= \bb{D}_{\bb{A}^4} \big( \pl(\der) \big) \cong \bb{A}^2 \times \big( \bb{A}^2 \setminus \{(0, 0)\} \big)$ is not an $\alpha$-cylinder. Denote $V = \bb{D}_{\bb{A}^4}(u)$. An isomorphism $U \cong \bb{A}^1 \times Z$ making $U$ an $\alpha$-cylinder would define a level function, which is regular on $\bb{A}^4$ since the codimension of $\bb{A}^4 \setminus U$ is greater than one. This level function restricts to a slice of $\der$ on $\bb{K}[V] = \bb{K}[x, y, u, u^{-1}, v]$. Every such slice has the form
$$
\frac{x}{u} + h(u, \ u^{-1}, \ v, \ x v - y u)
$$
for some polynomial $h$. This is a contradiction with regularity of the level function.
\end{exam}

%%%%%%%%%%%%%%%%%%%%%%%%%%%%%%

{}

%%%%%%%%%%%%%%%%%%%%%%%%%%%%%%

\end{document}